\documentclass[12pt,a4paper]{amsart}
\usepackage{amsfonts,amsmath}
\usepackage{enumerate}
\numberwithin{equation}{section}

\theoremstyle{plain}
\newtheorem{Th}{Theorem}
\newtheorem{Lemma}{Lemma}
\newtheorem{Cor}[Th]{Corollary}
\newtheorem{Prop}[Th]{Proposition}

  \theoremstyle{definition}
 \newtheorem{Def}[Th]{Definition}
\newtheorem*{remark}{Remark}

\newcommand{\setZ}{\mathbb Z}

\newcommand{\veps}{\varepsilon}

\begin{document}

\title{Difference sets and frequently hypercyclic weighted shifts}

\author{Fr\'ed\'eric Bayart, Imre Z. Ruzsa}

\address{Clermont Universit\'e, Universit\'e Blaise Pascal, Laboratoire de Mathe\-ma\-ti\-ques, BP 10448, F-63000 CLERMONT-FERRAND -
CNRS, UMR 6620, Laboratoire de Math\'emati\-ques, F-63177 AUBIERE
}
\address{Alfr\'ed R\'enyi Institute of Mathematics\\
     Budapest, Pf. 127\\
     H-1364 Hungary
}
\email{Frederic.Bayart@math.univ-bpclermont.fr}
\email{ruzsa@renyi.hu}



\begin{abstract}
We solve several problems on frequently hypercyclic operators.
Firstly, we characterize frequently hypercyclic weighted shifts on $\ell^p(\mathbb Z)$, $p\geq 1$.
Our method uses properties of the difference set of a set with positive upper density.
Secondly, we show that there exists  an operator 
which is $\mathcal U$-frequently hypercyclic, yet not frequently hypercyclic
and that there exists an operator which is frequently hypercyclic, yet not distributionally chaotic.
These (surprizing) counterexamples are given by weighted shifts on $c_0$. 
The construction of these shifts lies on the construction of sets of positive integers whose difference sets
have very specific properties.
\end{abstract}
\maketitle

\section{Introduction}
Let $X$ be a Banach space and let $T\in\mathfrak L(X)$ be a bounded operator on $X$. $T$ is called \emph{hypercyclic}
provided there exists a vector $x\in X$ such that its orbit $O(x,T)=\{T^n x;\ n\geq 0\}$ is dense in $X$.
$x$ is then called a hypercyclic vector for $T$ and we shall denote by $HC(T)$ the set of $T$-hypercyclic vectors.
The study of hypercyclic operators is a branch of linear dynamics, a very active field of analysis.
We refer to the books \cite{BM09} and \cite{GePeBOOK} to learn more on this subject.

In 2005, the first author and S. Grivaux have introduced in \cite{BAYGRITAMS} a refinement of the notion of hypercyclicity, called
\emph{frequent hypercyclicity}. For an operator to be frequently hypercyclic, one ask now that not only
there exists a vector with a dense orbit, but moreover that this orbit visits often each nonempty open subset.
To be more precise, let us introduce the following definitions and notations. For $A\subset \mathbb Z_+$,
we denote by $A(n)=\{a\in A;\ a\leq n\}$. The lower density of $A$ is defined by 
$$\underline d(A)=\liminf_{n\to+\infty}\frac{\# A(n)}{n},$$
and the upper density of $A$ is defined by
$$\bar d(A)=\limsup_{n\to+\infty}\frac{\# A(n)}{n}.$$
We will also use corresponding definitions and notations for subsets of $\mathbb Z$. For instance, for $A\subset \mathbb Z$ and $n\in\mathbb Z_+$, 
$A(n)=\{a\in A;\ |a|\leq n\}$. 

\begin{Def}
$T\in\mathfrak L(X)$ is called \emph{frequently hypercyclic} provided there exists a vector $x\in X$, called a \emph{frequently hypercyclic vector for $T$},  such that,
 for any $U\subset X$ open and nonempty, $\{n\in\mathbb Z_+;\ T^n x\in U\}$ has positive lower density. We shall denote by $FHC(T)$ the set
of frequently hypercyclic vectors for $T$. 
\end{Def}
This notion has then been investigated by several authors, see for instance \cite{BAYGRIPLMS}, \cite{BAYMATHERGOBEST}, \cite{BoGre07}, \cite{Gri11}, \cite{GEPE05},  \cite{Shk09}. In particular, it has many connections with ergodic theory.
Of course, one may also investigate the corresponding notion replacing lower by upper density, leading to the following definition introduced in \cite{Shk09}.

\begin{Def}
$T\in\mathfrak L(X)$ is called $\mathcal U$-\emph{frequently hypercyclic} provided there exists a vector $x\in X$, called a \emph{$\mathcal U$-frequently hypercyclic vector for $T$}, such that,
 for any $U\subset X$ open and nonempty, $\{n\in\mathbb Z_+;\ T^n x\in U\}$ has positive upper density. The set of $\mathcal U$-frequently hypercyclic vectors for $T$
will be denoted by $\mathcal UFHC(T)$.
\end{Def}

Several basic problems remain open regarding these two refinements of hypercyclicity. In this paper, we solve several of these problems. 
We begin by studying the frequently (resp. the $\mathcal U$-frequently) hypercyclic weighted shifts on $\ell^p$, $p\geq 1$.
Let $\mathbf w=(w_n)_{n\in\mathbb Z}$ be a bounded sequence of positive real numbers
and let $p\geq 1$. The bilateral weighted shift $B_{\mathbf w}$ on $\ell^p(\mathbb Z)$ is defined
by $B_{\mathbf w}(e_n)=w_ne_{n-1}$, where $(e_n)$ is the standard basis of $\ell^p(\mathbb Z)$. 
Hypercyclicity of $B_{\mathbf w}$ has been characterized by H. Salas in \cite{Sal95}. 
In Section 3, we will prove the following characterization of frequently hypercyclic weighted shifts.
\begin{Th}\label{THM1}
Let $p\in[1,+\infty)$ and let $\mathbf w=(w_n)_{n\in\mathbb Z}$ be a bounded sequence of positive real numbers. 
The following assertions are equivalent.
\begin{enumerate}[(i)]
 \item $B_{\mathbf w}$ is frequently hypercyclic on $\ell^p(\mathbb Z)$;
 \item $B_{\mathbf w}$ is $\mathcal U$-frequently hypercyclic on $\ell^p(\mathbb Z)$;
 \item The series $\sum_{n\geq 1}\frac{1}{(w_1\cdots w_n)^p}$ and $\sum_{n<0}(w_{-1}\cdots w_{n})^p$ are convergent.
\end{enumerate}
\end{Th}
A similar result holds for the unilateral weighted shift $B_{\mathbf w}$ on $\ell^p(\mathbb Z_+)$,
which is defined by $B_{\mathbf w}(e_n)=w_ne_{n-1}$ for $n\geq 1$ and by $B_{\mathbf w}(e_0)=0$. 
\begin{Th}\label{THM2}
Let $p\in[1,+\infty)$ and let $\mathbf w=(w_n)_{n\in\mathbb Z_+}$ be a bounded sequence of positive real numbers. 
The following assertions are equivalent.
\begin{enumerate}[(i)]
 \item $B_{\mathbf w}$ is frequently hypercyclic on $\ell^p(\mathbb Z_+)$;
 \item $B_{\mathbf w}$ is $\mathcal U$-frequently hypercyclic on $\ell^p(\mathbb Z_+)$;
 \item The series $\sum_{n\geq 1}\frac{1}{(w_1\cdots w_n)^p}$ is convergent.
\end{enumerate}
\end{Th}
That $\sum_{n\geq 1}\frac1{(w_1\cdots w_n)^p}<+\infty$ implies the frequent hypercyclicity of $B_{\mathbf w}$ on $\ell^p(\mathbb Z_+)$
is known since \cite{BAYGRITAMS}.
 A necessary condition for $B_{\mathbf w}$ to 
be frequently hypercyclic was given in \cite{BAYGRITAMS} and this condition was improved in \cite{GEPE05}. 
This last condition is the starting point of the present work. We show how to combine this condition with a result on the difference set of a set with positive upper
density to prove that (ii) implies (iii) in Theorem \ref{THM2}.

\smallskip

We then investigate frequently hypercyclic weighted shifts on $c_0$. It is worth noting that the previous theorems cannot be extended to $c_0$. 
Indeed,
in \cite{BAYGRIPLMS}, a frequently hypercyclic backward weighted shift $B_{\mathbf w}$ is exhibited on $c_0(\mathbb Z_+)$
such that the sequence $(w_1\cdots w_n)^{-1}$ \emph{does not} converge to 0 (in fact, one may require
that $w_1\cdots w_n=1$ for infinitely many $n$). Nevertheless, we will give in Section \ref{SECCARACC0} a characterization of frequently
and $\mathcal U$-frequently hypercyclic weighted shifts on $c_0$. This characterization is necessarily more difficult than that of Theorem \ref{THM2}. However, 
it will be efficient to give later in the paper nontrivial examples and counterexamples of frequently hypercyclic weighted shifts on $c_0(\mathbb Z)$. 

\medskip

Interestingly, the weighted shift $B_{\mathbf w}$ constructed in \cite{BAYGRIPLMS} give several counterexamples in the
theory of frequently hypercyclic operators:
\begin{itemize}
\item $B_{\mathbf w}$ is frequently hypercyclic, yet neither chaotic nor topologically mixing;
\item $B_{\mathbf w}$ does not admit any nonzero invariant Gaussian measure.
\end{itemize}
In Section \ref{SECFHCC0} and \ref{SECFHCDC}, we show that weighted shifts on $c_0$ can help to solve further problems on frequently hypercyclic operators. 
For instance, there are no examples in the literature of $\mathcal U$-frequently hypercyclic
operators which are not frequently hypercyclic. Weighted shifts on $c_0(\mathbb Z_+)$ give an example.

\begin{Th}\label{THMFHCC0}
There exists a bounded sequence $\mathbf w=(w_n)_{n\geq 1}$ of positive real numbers such that 
\begin{itemize}
\item[(i)] $B_{\mathbf w}$ is $\mathcal U$-frequently hypercyclic on $c_0(\mathbb Z_+)$;
\item[(ii)] $B_{\mathbf w}$ is not frequently hypercyclic on $c_0(\mathbb Z_+)$.
\end{itemize}
\end{Th}

Another problem that weighted shifts on $c_0$ can solve is related to distributionally chaotic operators. 
The notion of distributional chaos was introduced by B. Schweizer and J. Sm\'ital in \cite{SS94}. Let $f:X\to X$ be a continuous map
on a metric space $X$. For each $x,y\in X$ and each $n\in\mathbb N$, the \emph{distributional function} $F_{xy}^n:\mathbb R_+\to[0,1]$ is defined by
$$F_{xy}^n(\tau):=\frac 1n\textrm{card}\big\{0\leq i\leq n-1;\ d(f^i(x),f^i(y))<\tau\big\}.$$
Moreover define
$$F_{xy}(\tau):=\liminf_{n\to+\infty}F_{xy}^n(\tau)\textrm{ and }F_{xy}^*(\tau):=\limsup_{n\to+\infty}F_{xy}(\tau).$$
$f$ is said \emph{distributionally chaotic} if there exists an uncountable set $\Gamma\subset X$ and $\veps>0$
such that, for every $\tau>0$ and each pair of distinct points $x,y\in\Gamma$, we have
$$F_{xy}(\veps)=0\textrm{ and }F_{xy}^*(\tau)=1.$$
When $f$ is a linear map acting on a Banach space $X$, the notion of distributional chaos is related to the existence of distributionally
irregular vectors.
\begin{Def}
 Given $T\in\mathfrak L(X)$ and $ \veps>0$, a vector $x\in X$ is a \emph{distributional irregular vector} for $T$ if there exists
$A,B\subset \mathbb N$ with $\bar d(A)=\bar d(B)=1$ such that
$$\lim_{n\to+\infty,\ n\in A}T^n x=0\textrm{ and }\lim_{n\to+\infty,\ n\in B}\|T^n x\|=+\infty.$$
\end{Def}
It is shown in \cite{BBMP13} that $T$ is distributionally chaotic if and only if $T$ admits a distributional irregular vector. 

In \cite{BBMP13}, the following question is asked: are there frequently hypercyclic operators which are not distributionally chaotic?
We answer this question thanks to weighted shifts on $c_0(\mathbb Z)$. 
\begin{Th}\label{THMFHCDC}
 There exists a frequently hypercyclic weighted shift on $c_0(\mathbb Z)$ which is not distributionally chaotic.
\end{Th}

We conclude this paper in Section \ref{SECOPEN} by miscellaneous problems and results on frequently hypercyclic operators.
In particular, we show that if $T$ is invertible and frequently hypercyclic, then $T^{-1}$ is $\mathcal U$-frequently
hypercyclic.

\medskip

Let us mention that a common feature of Theorems \ref{THM1} to \ref{THMFHCDC} is their interaction with additive number theory. 
To prove Theorems \ref{THM1} and \ref{THM2}, we need a property of the difference set of a set with positive upper
density. The proofs
of Theorems \ref{THMFHCC0} and \ref{THMFHCDC} need both the construction of big sets of integers
such that their difference sets are sparse enough.

\section{A result on difference set}
Let $A\subset\mathbb Z_+$ be a set with positive upper density. 
A well-known result of Erd\"os and Sark\"ozy (see for instance \cite{Rus73}) ensures that the difference set $A-A$ is syndetic
(namely it has bounded gaps). Our first result is a strenghtening of this property. Let us introduce, for $k\in\mathbb Z_+$,
$$B_k=A\cap(A-k).$$
Erd\"os and Sark\"ozy proved that the set of integers $k$ such that $B_k$ is nonempty is syndetic. We shall prove that the set
of integers $k$ such that $B_k$ has a big upper density is also syndetic. For convenience, we formulate this for subsets of $\mathbb Z$.
\begin{Th}\label{THMSYNDETIC}
Let $A\subset \mathbb Z$ be a set with positive upper density, let $\delta=\bar d(A)$ and
let $\varepsilon\in(0,1)$. 
For any $k\in\mathbb Z$, let $B_k=A\cap(A-k)$ and let $\delta_k=\bar d(B_k)$. Let also $F=\{k;\ \delta_k>(1-\varepsilon)\delta^2\}$.
Then $F$ is syndetic.
\end{Th}
\begin{proof}
We select a sequence $(n_i)$ such that
 \[ \#A(n_i)/(2n_i+1) \to \delta .\]
 Then we select by a usual diagonal procedure a subsequence $(m_i)$ of $(n_i)$ such that for every $k\in\mathbb Z$ the limit
  \[ \eta_k = \lim \# B_k(m_i)/(2m_i+1) \]
  exists. Observing that $B_{k}=-k+B_{-k}$, one knows that $\eta_k=\eta_{-k}$. 
  Moreover, $\eta_k\leq \delta_k$ and we shall in fact prove that 
   \[ \mathbf F = \{ k; \eta_k > (1-\varepsilon)\delta^2 \} \]
   is syndetic.
 Let $R$ be a (finite) set with the property that
      \[ \eta_{k-l} \leq (1-\varepsilon)\delta^2 \]
      for all $k,l\in R$, $k \neq l$. We will see that the cardinality of such sets is uniformly bounded. We set $r=\#R$ and we put
       \[ f(x) = \# \{ k\in R; x \in A-k \} .\]
       We have
 \[ f(x)^2 = \# \{ k,l\in R; x \in(A-k) \cap (A-l) \}    = \# \{ k,l\in R; x+k \in A \cap (A+k-l) \}       .\]
 Clearly
  \[ \sum_{|x| \leq m} f(x) = \sum_{k\in R}\#\big\{x\in\{-m,\dots,m\};\ x\in A-k\big\}=r \#A(m) + O(1), \]
  hence
 \[ \frac{1}{2m_i+1} \sum_{|x| \leq m_i} f(x) \to r\delta . \]
 Similarly
\[ \sum_{|x| \leq m} f(x)^2 = \sum_{k,l\in R} \#B_{k-l}(m) + O(1), \]
  hence
 \[ \frac{1}{2m_i+1} \sum_{|x| \leq m_i} f(x)^2 \to\sum_{k,l\in R} \eta_{k-l} \leq r\delta + (1-\varepsilon)r(r-1)\delta^2 . \]
 Using the inequality of arithmetic and square means we get
  \[ (r\delta)^2 \leq r\delta + (1-\varepsilon)r(r-1)\delta^2, \]
  hence
   \[ r \leq \frac{1-\delta(1-\varepsilon)}{\delta\varepsilon} .\]

   Now select a maximal set $R$ (take 0, then the integer $n$ with the smallest absolute value which can be added and so on).
  This procedure stops after a finite number of steps.
 Maximality means that for every  integer $n$ there is a $k\in R$ such that $\eta_{n-k}> (1-\varepsilon)\delta^2$, that is, $n-k\in \mathbf F$,
which means that $\mathbf F+R=\setZ$. This amounts to say that $\mathbf F$ is syndetic.
\end{proof}
\begin{remark}
 The previous theorem is reminiscent from Khintchine's recurrence theorem which says the following: for any invertible
probability measure preserving system $(X,\mathcal B,\mu,T)$, for any $\varepsilon>0$ and any $A\in\mathcal B$, the set
$\{n\in\mathbb Z;\ \mu(A\cap T^nA)\geq \mu(A)^2-\varepsilon\}$ is syndetic. It turns out that we may deduce
Theorem \ref{THMSYNDETIC} from Khintchine's recurrence theorem using Furstenberg's correspondence principle
(see \cite{Fur77}), exactly as Furstenberg deduced the Szemer\'edi's theorem on arithmetic progressions
from his extension of the classical Poincar\'e's recurrence theorem (we refer to  \cite{Ber96} and to \cite{Fur77} 
for details).

One may prove Khintchine's recurrence theorem using the uniform version of von Neumann's ergodic theorem. One can also find in 
\cite{Ber96} a combinatorial proof of this theorem, which does not match exactly the proof of Theorem \ref{THMSYNDETIC}.
To keep a self-contained exposition, we have chosen to give a complete and elementary proof of Theorem \ref{THMSYNDETIC}.
\end{remark}

From this, we can deduce a result on series which is the key for the application to frequently hypercyclic weighted shifts.
\begin{Cor}\label{CORSERIES}
Let $(\alpha_n)_{n\in\mathbb Z}$ be a sequence of nonnegative real numbers such that $\sum_n \alpha_n=+\infty$.
Suppose that there exists some $C>0$ such that either $\alpha_n\geq C\alpha_{n-1}$ for every $n\in\mathbb Z$ or $\alpha_{n-1}\geq C\alpha_n$ for every $n\in\mathbb Z$.
Let $A\subset\setZ$ be a set with positive upper density and let, for $n\in A$, 
$$\beta_n=\sum_{m\in A}\alpha_{m-n}.$$
Then 
$$\limsup_{n\to+\infty}\frac 1{2n+1}\sum_{|m|\leq n,m\in A}\beta_m=+\infty.$$
\end{Cor}
In particular, the sequence $(\beta_n)_{n\in A}$ cannot be bounded.
\begin{proof}
We keep the notations of the previous theorem, which we apply with $\varepsilon=1/2$. We first show that
\begin{eqnarray}\label{EQCOR1}
\sum_{n\in F}\alpha_n=+\infty,
\end{eqnarray}
using only that $F$ is syndetic.
Indeed, write $F=(f_n)_{n\in \setZ}$ in increasing order with $f_0=\min\{f\in F;\ f\geq 0\}$. 
There exists some $M>0$ such that $f_{i+1}-f_i\leq M$ for every $i\in\setZ$. 
Assuming first that 
$\alpha_n \geq C \alpha_{n-1}$ for every $n$, we get
      \[ \alpha_{f_j} \geq\frac{\max(1,C^M)}{M} \sum_{f_{j-1}<i\leq f_j} \alpha_i, \]
      and (\ref{EQCOR1}) follows by summing this for all $j$. If $\alpha_{n-1}\geq C\alpha_n$ for every $n$, then we write
            \[ \alpha_{f_j} \geq\frac{\max(1,C^M)}M \sum_{f_{j}\leq i< f_{j+1}} \alpha_i, \]
and (\ref{EQCOR1}) follows also by summation.
Now, consider the sum
      \[ s_i=  \sum_{|n|\leq m_i, n \in A} \beta_n .     \]
      This can be rewritten as
       \[ s_i=  \sum_{|n|\leq m_i,  n,m \in A} \alpha_{m-n} .     \]
      We 
      group this sum according to the value of $k=m-n$, and keep only those terms where $k\in F$, $|k|<l$ for some fixed $l$.
      We get
       \[ s_i \geq \sum_{k\in F, |k|<l} \alpha_k \#B_k(m_i) . \]
       We divide by $2m_i+1$ and let $i\to\infty$. We get
       \begin{eqnarray*}
        \limsup_{i\to+\infty} \frac{1}{2m_i+1} \sum_{|n|\leq m_i, n \in A} \beta_n &\geq & \sum_{k\in F, |k|<l} \alpha_k
	\lim_{i\to+\infty} \frac{\#B_k(m_i)}{2m_i+1} \\
        &\geq& \sum_{k\in F, |k|<l} \alpha_k \eta_k \\
	&\geq& \frac{\delta^2}{2}  \sum_{k\in F, |k|<l} \alpha_k ,  
	\end{eqnarray*}
        and this can be arbitrarily large by (\ref{EQCOR1}).

   \end{proof}
   
\section{Frequently hypercyclic weighted shifts on $\ell^p$} \label{SECWS}

In this section, we prove Theorem \ref{THM1}. The proof of Theorem \ref{THM2} is similar but simpler. 
We first prove that $(ii)$ implies $(iii)$. Thus we start with a $\mathcal U$-frequently hypercyclic weighted shift
$B_{\mathbf w}$ on $\ell^p(\mathbb Z)$ and let $x$ be a $\mathcal U$-frequently hypercyclic vector for $B_{\mathbf w}$. 
Let 
$$A=\big\{n\in\mathbb Z_+;\ \|B_{\mathbf w}^n x-e_0\|_p\leq 1/2\big\},$$
which has positive upper density.
Let $m\in A$. Then $|w_1\cdots w_m x_m-1|\leq 1/2$ so that $|w_1\cdots w_m x_m|\geq 1/2$. 
Now, for any $n\in A$, we can also write
\begin{eqnarray*}
\frac1{2^p}&\geq& \|B_{\mathbf w}^n x-e_0\|^p_p\\
&\geq& \sum_{m\in A,m<n}(w_m\cdots w_1 w_0\cdots w_{m-n+1})^p |x_m|^p+\sum_{m\in A,m>n}(w_m\cdots w_{m-n+1})^p |x_m|^p\\
&\geq&\sum_{m\in A,m<n}(w_0\cdots w_{m-n+1})^p|w_1\cdots w_mx_m|^p+\sum_{m\in A,m>n}\frac{|w_1\cdots w_mx_m|^p}{(w_1\cdots w_{m-n})^p}.
\end{eqnarray*}
Putting this together, we get that for any $n\in A$,
$$\left\{
\begin{array}{rcl}
\displaystyle \sum_{m\in A,m<n}(w_0\cdots w_{m-n+1})^p&\leq& 1\\
\displaystyle \sum_{m\in A,m>n}\frac1{(w_1\cdots w_{m-n})^p}&\leq&1.
\end{array}\right.$$
Firstly, we set $\alpha_n=0$ provided $n\leq 0$ and $\alpha_n=\frac1{(w_1\cdots w_n)^p}$ provided $n>0$.
Because $(w_n)_{n\in\mathbb Z}$ is bounded, $\alpha_n\geq C\alpha_{n-1}$ for every $n\in\mathbb Z$.
Suppose that $\sum_{n\geq 1}\frac{1}{(w_1\cdots w_n)^p}=+\infty$. Then by Corollary \ref{CORSERIES},
the sequence $(\beta_n)_{n\in A}$ is unbounded, where 
$$\beta_n=\sum_{m\in A}\alpha_{m-n}=\sum_{m\in A,m>n}\alpha_{m-n}=\sum_{m\in A,m>n}\frac1{(w_1\cdots w_{m-n})^p}.$$
This is a contradiction. Secondly, set $\alpha_n=0$ provided $n\geq 0$ and $\alpha_n=(w_0\cdots w_{-n+1})^p$ provided $n<0$. 
Because $(w_n)_{n\in\mathbb Z}$ is bounded, $\alpha_n\leq C\alpha_{n-1}$ for any $n\in\mathbb Z$. Suppose that
$\sum_{n<0}(w_0\cdots w_n)^p=+\infty$. Then by Corollary \ref{CORSERIES}, $(\beta_n)$ is unbounded where 
$$\beta_n=\sum_{m\in A}\alpha_{m-n}=\sum_{m\in A,m<n}\alpha_{m-n}=\sum_{m\in A,m<n}(w_0\cdots w_{m-n})^p.$$
This is also a contradiction, since $w_0\cdots w_{m-n}\leq C w_0\cdots w_{m-n+1}$.

\smallskip

Let us now show that the condition is sufficient. This follows from a standard application of the frequent hypercyclicity criterion of \cite{BoGre07}, which we recall for convenience:
\begin{Th}
Let $T\in\mathfrak L(X)$, where $X$ is a separable Banach space. Assume that there exists a dense set $\mathcal D\subset X$ and a map
$S:\mathcal D\to\mathcal D$ such that 
\begin{enumerate}
\item $\sum T^n(x)$ and $\sum S^n(x)$ converge unconditionally for any $x\in\mathcal D$;
\item $TS=I$ on $\mathcal D$.
\end{enumerate}
Then $T$ is frequently hypercyclic.
\end{Th}
In our situation, we define $S$ by $S(e_n)=w_{n+1}^{-1}e_{n+1}$ and let $\mathcal D$ be the set of finitely supported sequences.
It is easy to check that $\sum_{n<0}(w_{-1}\cdots w_{-n})^p<+\infty$ implies that
the series $\sum_n B_{\mathbf w}^n x$ is unconditionally convergent for any $x\in\mathcal D$.
In the same vein, the condition $\sum_{n\geq 1}\frac{1}{(w_1\cdots w_n)^p}<+\infty$ implies that
$\sum_n S^n x$ is unconditionally convergent for any $x\in\mathcal D$. Thus, $B_{\mathbf w}$ is frequently hypercyclic.

\medskip

Our result implies the following interesting corollary.
\begin{Cor}
Let $\mathbf w=(w_n)_{n\in\mathbb Z}$ be a bounded and bounded below sequence of positive real numbers.
Then $B_{\mathbf w}$ is frequently hypercyclic on $\ell^p(\mathbb Z)$ if and only if $B_{\mathbf w}^{-1}$ is frequently hypercyclic on $\ell^p(\mathbb Z)$.
\end{Cor}

\section{Frequently hypercyclic weighted shifts on $c_0$}\label{SECCARACC0}

In this section, we give a characterization of (invertible) frequently hypercyclic weighted shifts on $c_0(\mathbb Z_+)$
and $c_0(\mathbb Z)$. Because of \cite{BAYGRIPLMS}, we know that we cannot expect a statement as clean as Theorem \ref{THM2}.
Nevertheless, it will be useful in the forthcoming examples. We begin with invertible bilateral weighted shifts.
\begin{Th}\label{THMCARACBILATERAL}
 Let $\mathbf w=(w_n)_{n\in\mathbb Z}$ be a bounded and bounded below sequence of positive integers. Then 
$B_{\mathbf w}$ is frequently hypercyclic (resp. $\mathcal U$-frequently hypercyclic) on $c_0(\mathbb Z)$ if and only if there exist
a sequence $(M(p))$ of positive real numbers tending to $+\infty$ and a sequence $(E_p)$ of subsets of $\mathbb Z_+$
such that
\begin{enumerate}[(a)]
 \item For any $p\geq 1$, $\underline d(E_p)>0$ (resp. $\bar d(E_p)>0$);
\item For any $p,q\geq 1$, $p\neq q$, $(E_p+[-p,p])\cap (E_q+[-q,q])=\emptyset$;
\item $\lim_{n\to+\infty,\ n\in E_p}w_1\cdots w_n=+\infty$;
\item For any $p,q\geq 1$, for any $n\in E_p$ and any $m\in E_q$ with $n\neq m$, 
$$\left\{\begin{array}{rcll}
                   w_1\cdots w_{m-n}&\geq&M(p)M(q)&\textrm{ provided } m>n\\
w_{m-n+1}\cdots w_0&\leq&\displaystyle \frac1{M(p)M(q)}&\textrm{ provided }m<n.
                  \end{array}\right.$$
\end{enumerate}
\end{Th}
\begin{proof}
 We first observe that we may replace ``there exists a sequence $(M(p))$'' by ``for any sequence $(M(p))$'' in the statement of
the previous theorem. Indeed, if properties (a) to (d) are true for some sequence $(M(p))$, then they are also satisfied for any sequence
$(M(p))$, considering instead of $(E_p)$ a subsequence of $(E_p)$ if necessary. 

We just prove the frequently hypercyclic case, the $\mathcal U$-frequently hypercyclic one being completely similar. We first assume that
$B_{\mathbf w}$ is frequently hypercyclic and we let $x\in FHC(B_{\mathbf w})$.
Let us fix $\rho>1$ such that $\rho^{-1}\leq w_k\leq \rho$ for any $k\in\mathbb Z$. Let us also consider a sequence $(\omega_p)$
of positive real numbers such that $\omega_1=2$ and, for any $p\geq 2$, $\omega_p>4\omega_{p-1}\rho^{2p+1}$. We set 
$$E_p=\left\{n\in\mathbb Z_+;\ \left\|B_{\mathbf w}^n x-\omega_p(e_{-p}+\dots+e_p)\right\|_{\infty}<\frac 1p\right\}.$$
Since $x$ belongs to $FHC(B_{\mathbf w})$, $E_p$
has positive lower density. Let $p\neq q$ and let us show that $(E_p+[-p,p])\cap(E_q+[-q,q])=\emptyset$. By contradiction, 
let us assume that $(n,s,m,t)\in E_p\times[-p,-p]\times E_q\times[-q,q]$ with $n+s=m+t$. Without loss of generality,
we may assume $p<q$. $w_{s+1}\cdots w_{n+s}x_{n+s}$ is the
$s$-th coefficient of $B_{\mathbf w}^n x$. Its modulus is smaller than $2\omega_p$. Similarly, $w_{t+1}\cdots w_{m+t}x_{m+t}$ 
is the $t$-th coefficient of $B_{\mathbf w}^m x$. Its modulus is greater than $\omega_q/2$. Moreover, $w_{s+1}\cdots w_{n+s}x_{n+s}$ and
$w_{t+1}\cdots w_{m+t}x_{m+t}$ differ by at most $(2q+1)$ coefficients of the sequence $\mathbf w$. 
Hence,
\begin{eqnarray}
 \label{EQCARACFHC3}
\left(\frac 1\rho\right)^{2q+1}\leq \frac{w_{s+1}\cdots w_{n+s}|x_{n+s}|}{w_{t+1}\cdots w_{m+t}|x_{m+t}|}\leq 2\omega_p\times\frac 2{\omega_q}.
\end{eqnarray}
This contradicts the definition of $(\omega_n)$. 

Moreover, pick $n\in E_p$ and look at the $0$-th coefficient of $B_{\mathbf w}^n x$. It is equal to $w_1\cdots w_n x_n$ and its modulus cannot be less than $\omega_p/2$.
Since $x\in c_0(\mathbb Z)$, we get that $w_1\cdots w_n$ tends to $+\infty$ when $n$ goes to infinity, $n\in E_p$. Fix another $m\in E_q$, $m\neq n$
and look at the $(n-m)$-th coefficient of $B_{\mathbf w}^m x$. This coefficient is equal to $w_{n-m+1}\cdots w_n x_n$ and its modulus is less than $1/q$
(recall that $|n-m|>q$). If $n>m$, then since $w_1\cdots w_n x_n\geq\omega_p/2$, we can deduce
$$w_1\cdots w_{n-m}\geq \frac{w_1\cdots w_n |x_n|}{w_{n-m+1}\cdots w_n |x_n|}\geq q\frac{\omega_p}2.$$
Similarly, if $n<m$, then
$$w_{n-m+1}\cdots w_0\leq \frac{w_{n-m+1}\cdots w_n |x_n|}{w_1\cdots w_n |x_n|}\leq \frac{2}{q\omega_p}.$$
This shows (d) with $M(p)=p$.

\medskip

We now show that the condition is sufficient. As pointed out above, we may assume that, for any $p\geq 1$, $M(p)\geq \rho^{4p}$. We set
$$E'_p=E_p\backslash \big\{n\in\mathbb N;\ w_1\cdots w_n\leq \rho^{4p}\big\}.$$
$E'_p$ is a cofinite subset of $E_p$, hence it has positive lower density. We write $E'_p=(n_k^p)_{k\geq 0}$ in an increasing order
and we set $F_p=(n_{(2p+1)k}^p)_{k\geq 0}$. $F_p$ has positive lower density and $|n-m|>2p$ provided $n,m$ are two
distinct elements of $F_p$. 

Let $(y(p))_{p\geq 1}$ be a dense sequence 
in $c_0(\mathbb Z)$ such that the support of $y(p)$ is contained in $[-p,-p]$ and such that $\|y(p)\|_{\infty}\leq\rho^p.$
We define $x\in\mathbb C^{\mathbb N}$ by setting
$$x_k=\left\{
\begin{array}{ll}
\frac 1{w_{s+1}\cdots w_{n+s}}y_p(s)&\textrm{ for }k=n+s,\ n\in F_p,\ |s|\leq p\\
0&\textrm{ otherwise.}
\end{array}\right.$$
This definition is not ambiguous because of (b) and the definition of $F_p$. We claim that $x$ belongs to $c_0(\mathbb Z)$. Indeed, let $\veps>0$. 
For $p\geq 1$ and $n\in F_p$, $|s|\leq p$,
\begin{eqnarray}
 \label{EQCARACFHC1}
|x_k|&\leq&\frac{\rho^{2p}}{w_1\cdots w_n}\times\rho^p\leq \rho^{-p}\leq\veps
\end{eqnarray}
provided $p$ is greater than some $p_0\geq 1$. Now, fix $p\leq p_0$. Then by (c), 
$x_k$ goes to zero when $k$ goes to $+\infty$, $k$ staying in $F_p+[-p,-p]$. 

We then show that $x$ is a frequently hypercyclic vector for $B_{\mathbf w}$. It is sufficient to prove that,
for any $p\geq 1$ and any $n\in F_p$, $\|B_{\mathbf w}^n x-y(p)\|_\infty\leq\veps(p)$ with $\veps(p)\to 0$ as
$p$ goes to $+\infty$. We observe that
$$\|B_{\mathbf w}^n x-y(p)\|_{\infty}=\sup_{s\notin [-p,p]}|w_{s+1}\cdots w_{n+s}x_{n+s}|.$$
The terms which appear in the sup-norm are nonzero if and only if $n+s=m+t$, for some $m\in E_q$, $q\geq 1$,
and $t\in [-q,q]$. We distinguish two cases. First, if $m>n$, then we write
$$w_{s+1}\cdots w_{n+s}x_{n+s}=\left\{
 \begin{array}{ll}
 \displaystyle \frac{w_1\cdots w_t}{w_1\cdots w_{m-n+t}}y_t(q)&\textrm{if }t\geq 1\\[3mm]
 \displaystyle \frac{1}{w_{t+1}\cdots w_{0}}\times\frac 1{w_1\cdots w_{m-n+t}}y_t(q)&\textrm{if }t\leq 0.
 \end{array}
 \right.$$
Now, $w_1\cdots w_t\leq \rho^q$ if $t\geq 0$, $(w_{t+1}\cdots w_0)^{-1}\leq \rho^q$ if $t<0$, so that in both cases
\begin{eqnarray}
 \label{EQCARACFHC2}
|w_{s+1}\cdots w_{n+s}x_{n+s}|\leq \frac{\rho^{2q}}{w_1\cdots w_{m-n+t}}\leq\frac{\rho^{2q}\rho^p}{w_1\cdots w_{m-n}}\leq 
\frac{\rho^{2q}\rho^p}{\rho^{4q}\rho^{4p}}\leq \rho^{-3p}.
\end{eqnarray}
Second, if $m<n$, then we write
\begin{eqnarray*}
 w_{s+1}\cdots w_{n+s}x_{n+s}&=&w_{m-n+t+1}\cdots w_t w_{t+1}\cdots w_{m+t}x_{m+t}\\
&=&w_{m-n+t+1}\cdots w_t y_t(q)\\
&=&\left\{
\begin{array}{ll}
 \displaystyle w_{m-n+t+1}\cdots w_0w_1\cdots w_t y_t(q)&\textrm{if }t\geq 1\\[1mm]
\displaystyle \frac{w_{m-n+t+1}\cdots w_0}{w_{t+1}\cdots w_0}y_t(q)&\textrm{if }t\leq 0.
\end{array}\right.
\end{eqnarray*}
We conclude as before.
\end{proof}

We turn to unilateral weighted shifts. A similar statement holds.
\begin{Th}\label{THMCARACUNILATERAL}
 Let $\mathbf w=(w_n)_{n\in\mathbb Z_+}$ be a bounded sequence of positive integers. Then 
$B_{\mathbf w}$ is frequently hypercyclic (resp. $\mathcal U$-frequently hypercyclic) on $c_0(\mathbb Z_+)$ if and only if there exist
a sequence $(M(p))$ of positive real numbers tending to $+\infty$ and a sequence $(E_p)$ of subsets of $\mathbb Z_+$
such that
\begin{enumerate}[(a)]
 \item For any $p\geq 1$, $\underline d(E_p)>0$ (resp. $\bar d(E_p)>0$);
\item For any $p,q\geq 1$, $p\neq q$, $(E_p+[0,p])\cap (E_q+[0,q])=\emptyset$;
\item $\lim_{n\to+\infty,\ n\in E_p+[0,p]}w_1\cdots w_n=+\infty$;
\item For any $p,q\geq 1$, for any $n\in E_p$ and any $m\in E_q$ with $m>n$, 
for any $t\in\{0,\dots,q\}$,
$$w_1\cdots w_{m-n+t}\geq M(p)M(q).$$
\end{enumerate}
\end{Th}
\begin{proof}
The proof is more or less a rephrasing of the proof of Theorem \ref{THMCARACBILATERAL}. We have to take into account that $\mathbf w$ is not necessarily bounded below. This was used at several places:
\begin{itemize}
\item to prove that $x$ belongs to $c_0$; this remains true because we have a stronger assumption (c).
\item to obtain inequalities (\ref{EQCARACFHC3}),  (\ref{EQCARACFHC1}) and (\ref{EQCARACFHC2}). This is settled by the stronger assumption (d)
and by adjusting the values of $\omega_p$ and $M(p)$. For instance, we may choose
$$\omega_p\geq 4\omega_{p-1}\rho^{p+1}\times\frac1{\min(1,\inf (w_t^{p+1};\ t\in[0,p]))}$$
$$M(p)\geq \rho^{4p}\times\frac1{\min(1,\inf (w_t^{2p};\ t\in[0,p]))}.$$
\end{itemize}
The details are left to the reader.
\end{proof}

\section{A $\mathcal U$-frequently hypercyclic operator which is not frequently hypercyclic}\label{SECFHCC0}
We turn to the proof of Theorem \ref{THMFHCC0}. It requires careful constructions. We first build sequences of integers
with positive upper density and additional properties. These sequences allow us to define our weight $\mathbf w$. We then conclude
by showing that $B_{\mathbf w}$ is not frequently hypercyclic and by applying Theorem \ref{THMCARACUNILATERAL} to show that $B_{\mathbf w}$
is $\mathcal U$-frequently hypercyclic.
The rest of this section is devoted to these constructions. 

\subsection{The sequences of integers}
We shall construct sets of integers  $(E_p)_{p\geq 1}$ and sequences of integers $(a_r)_{r\geq 1}$, $(b_r)_{r\geq 1}$ satisfying the 
following properties:
\begin{description}
\item[(S1)] For any $r\geq 1$, $a_{r+1}\geq b_r+2r+1$, $b_r\geq ra_r$ and $b_r>r^2(2r+1)$;
\item[(S2)] For any $p\geq 1$, $\overline d\big(E_p\big)>0$;
\item[(S3)] For any $p\geq 1$, $E_p\subset {b_p}\mathbb N$;
\item[(S4)] For any $p,q\geq 1$ with $p\neq q$ and any $(n,m)\in E_p\times E_q$ with $m>n$, then  
\begin{eqnarray*}
m-n&>&p\\
m-n&\notin&\bigcup_{r\geq 1}[a_r-(r+1)-q;b_r+r+p+q];\\
n&\notin&\bigcup_{r\geq 1}[a_r-(r+1)-p;b_r+2r].
\end{eqnarray*}
\end{description}
The construction of these sequences is done by induction. Precisely, at Step $r$, we construct
integers $a_r$, $b_r$, $N_{p,r}$ for $p\leq r$ and subsets $E_p^r$ of ${b_p}\mathbb N$ for $p\leq r$ such that
\begin{itemize}
\item $a_r\geq b_{r-1}+2(r-1)+1$, $b_r\geq ra_r$ and $b_r>r^2(2r+1)$;
\item For any $p\leq r$, 
$$\displaystyle \#E_p^r(N_{p,r})\geq \frac1{2b_p}N_{p,r}.$$
\item For any $p,q\in\{1,\dots,r\}$ with $p\neq q$ and any $(n,m)\in E_p^r\times E_q^r$ with $m>n$, then  
\begin{eqnarray*}
m-n&>&p\\
m-n&\notin&\bigcup_{\rho=1}^r[a_{\rho}-({\rho}+1)-q;b_{\rho}+{\rho}+p+q];\\
n&\notin&\bigcup_{\rho=1}^r[a_{\rho}-({\rho}+1)-p;b_{\rho}+2{\rho}].
\end{eqnarray*}
\item For any $p<r$, $E_{p}^{r-1}\subset E_p^r$.
\end{itemize}
Provided this construction has been done, it is enough to set $E_p=\bigcup_{r\geq p}E_p^r$. The initialization of the induction is very easy.
One just sets for instance $a_1=1$, $b_1=4$, $N_{1,1}=8$ and $E_1^1=\{8\}$. 
Let us explain how to proceed with Step $r+1$ provided the construction has been done until Step $r$. Let $a_{r+1}$ be any integer
such that 
$$a_{r+1}\geq\left\{
\begin{array}{l}
b_r+2r+1\\
\max\big(n+p;\ p\leq r,\ n\in E_p^r\big)+(r+2).
\end{array}
\right.$$
Next we set $b_{r+1}=\max\big((r+1)a_{r+1},(r+1)^2 (2r+3)+1\big)$. In particular, it is clear that if $(n,m)\in E_p^r\times E_q^r$ with $m>n$ and $p\neq q\in\{1,\dots,r\}$, then
\begin{eqnarray*}
m-n&>&p\\
m-n&\notin&\bigcup_{\rho=1}^{r+1}[a_{\rho}-({\rho}+1)-q;b_{\rho}+{\rho}+p+q];\\
n&\notin&\bigcup_{\rho=1}^{r+1}[a_{\rho}-({\rho}+1)-p;b_{\rho}+2{\rho}].
\end{eqnarray*}
Let us now define $E_{1}^{r+1}$. We first set
$$M_{1,r+1}=b_{r+1}+3(r+1)+\max(E_p^r;\ p\leq r)$$
and we consider $N_{1,r+1}\geq M_{1,r+1}$ such that
$$\#\big([M_{1,r+1};N_{1,r+1}]\cap {b_1}\mathbb N)\geq \frac 1{2b_1}N_{1,r+1}.$$
We then set $E_{1}^{r+1}=E_{1}^{r}\cup ([M_{1,r+1};N_{1,r+1}]\cap {b_1}\mathbb N)$. 
In particular, if $m$ belongs to $E_{1}^{r+1}\backslash E_1^r$ and $n\in E_p^r$ with $m>n$ and $2\leq p\leq r$, then 
$$
\left\{
\begin{array}{rcl}
m-n&\geq& b_{r+1}+3(r+1)\geq b_{r+1}+(r+1)+p+1> p\\
m&\geq &b_{r+1}+3(r+1)>b_{r+1}+2(r+1).
\end{array}\right.
$$
The construction of the other sets $E_{p}^{r+1}$, $2\leq p\leq r+1$, follows exactly the same lines, defining first
$$M_{p,r+1}=b_{r+1}+3(r+1)+\max(E_{q}^m;\ 1\leq q\leq m\leq r\textrm{ or } 1\leq q\leq p-1, m=r+1).$$
The remaining details are left to the reader. We point out that, for $p\neq q$, 
$E_p+[0,p]$ does not intersect $E_q+[0,q]$. 

\subsection{The weight}\label{SUBSECWEIGHT}
We first define a weight $\mathbf w^0$ whose behaviour is adapted to the sequence $(a_r)$. Precisely, for $n\geq 1$, 
$w_n^0$ is defined by
\begin{itemize}
\item $w^0_n=2$ provided $n\notin \bigcup_{r\geq 1}[a_r-r;b_r+r]$;
\item $w^0_{a_r-r}$ is the (very small) positive real number such that $w_1^0\cdots w_{a_r-r}^0=1$;
\item $w_n^0=1$ otherwise.
\end{itemize}
The main interest of $\mathbf w^0$ is that the product $w_1^0\cdots w_n^0$ is rather large when $n$ belongs to a
difference set $E_p-E_q$, with $p\neq q$, or to a set $E_p$, whereas $w_1^0\cdots w_n^0=1$ if $n\in \bigcup_{r\geq 1}[a_r-r;b_r+r]$.
We then define, for each $p\geq 1$, a weight $\mathbf w^p$ which is suitable for the difference set $E_p-E_p$. Indeed, let
\begin{itemize}
\item $w_n^p=2$ provided $n={b_p}k+u$, with $k\geq 1$ and $u\in \{-(p-1),\dots,0\}$;
\item $w_n^p=\frac 1{2^{p}}$ provided $n={b_p}k+p+1$, $k\geq 1$;
\item $w_n^p=1$ otherwise.
\end{itemize}
This is not ambiguous since $b_p>2p+1$. Notice that
$w_1^p\cdots w_n^p=2^p$ provided $n\in {b_p}\mathbb N+[0,p]$ whereas $w_1^p\cdots w_n^p=1$ is equal to 1 outside 
$\bigcup_{k\geq 1}[{b_p}k-(p-1);{b_p}k+p]$. 

The weight $\mathbf w$ combines the properties of all $\mathbf w^p$. It is defined by setting by induction on $n\geq 1$
$$w_1\cdots w_n=\max(w_1^p\cdots w_n^p;\ p\geq 0).$$
$\mathbf w$ is well-defined. Indeed, let $n\geq 1$ and let $r\geq 1$ be such that $n\in [b_r,b_{r+1})$. 
Then $w_1^p\cdots w_n^p=1\leq w_1^0\cdots w_n^0$ provided $p\geq r+2$, so that 
${b_p}-p\geq b_{r+1}.$ Moreover $\mathbf w$ is bounded by $2$, since its definition easily implies that,
for any $n\geq 1$, $w_n\leq \max(w_n^p;\ p\geq 0)\leq 2$.

 We shall point out several important facts regarding $\mathbf w$ which come
from the properties of $\mathbf w^0$ and $\mathbf w^p$, $p\geq 1$.
The products $w_1\cdots w_n$ and $w_1\cdots w_{n-m}$ for $(n,m)\in E_p\times E_q$, $m>n$, are large.
Indeed, let $r$ be the unique integer such that $n\in [b_r+2r+1;a_{r+1}-(r+2)-p]$. Then, for any $s\in\{0,\dots,p\}$,
the definition of $\mathbf w$ ensures that
\begin{eqnarray}
w_1\cdots w_{n+s}&\geq& w_1^0\cdots w_{n+s}^0\nonumber\\
&\geq& w_1^0\cdots w_{b_r+r}^0w_{b_r+r+1}^0\cdots w_n^0w_{n+1}^0\cdots w_{n+s}^0\nonumber\\
\label{EQW1} &\geq &1\cdot 2^r\cdot 1.
\end{eqnarray}
Moreover, if $p\neq q$, there exists some $\rho\geq 1$ such that $m-n$ belongs
to $[b_\rho+\rho+p+q+1;a_{\rho+1}-(\rho+2)-q-1]$, so that, for any $t\in\{0,\dots,q\}$,
\begin{eqnarray}\label{EQW2}
w_1\cdots w_{m-n+t}\geq w_1^{0}\cdots w_{m-n+t}^0\geq 2^{p+q}.
\end{eqnarray}
If $p=q$, then 
\begin{eqnarray}\label{EQW3}
w_1\cdots w_{m-n+t}\geq w_1^{p}\cdots w_{m-n+t}^p\geq 2^{p}.
\end{eqnarray}
On the contrary, $w_1\cdots w_n$ is often small. Indeed, observe that for $p\geq 1$, 
$w_1^p\cdots w_n^p\leq 2^p$ for any $n>1$ and that $w_1^0\cdots w_n^0=1$ provided 
$n\in\bigcup_{r\geq 1} [a_r,b_r]$. Hence, if $n$ belongs to $\bigcup_{r\geq 1} [a_r,b_r]$ and satisfies $w_1\cdots w_n>2^p$, then
there exists $q>p$ such that $n\in {b_q}\mathbb N+[-q,q]$.

\subsection{$B_{\mathbf w}$ is not frequently hypercyclic}
Assume on the contrary that $B_\mathbf w$ is frequently hypercyclic. Then there exists $E\subset\mathbb N$
with $\underline{d}(E)>0$ such that $w_1\dots w_n\to +\infty$ when $n\to+\infty$, $n\in E$.
In particular, for any $p\geq 1$, 
$$F_p=\big\{n\in E;\ w_1\cdots w_n> 2^{p}\big\}$$
is a cofinite subset of $E$. It has the same lower density. Now, let $r\geq 1$ and let $n \in F_p\cap [0,b_r]$. Then
either $n\leq a_r$ or there exists $q>p$ such that $n$ belongs to  ${b_q}\mathbb N+[-q,q]$. This yields
$$\#F_p(b_r)\leq a_r+b_r\times \sum_{q>p}\frac{2q+1}{b_q}.$$
Since $a_r/b_r$ goes to zero, this implies
$$\underline{d}(E)=\underline{d}(F_p)\leq \sum_{q> p}\frac{1}{q^2}.$$
Since $p$ is arbitrary, $\underline{d}(E)=0$ and $B_{\mathbf w}$ cannot be frequently hypercyclic.

\subsection{$B_{\mathbf w}$ is $\mathcal U$-frequently hypercyclic}
This follows from an application of Theorem \ref{THMCARACUNILATERAL} for the sets $E_p$ defined above and from the work of Subsection \ref{SUBSECWEIGHT}. Indeed, condition (c)
of this theorem follows from (\ref{EQW1}) whereas condition (d) is a consequence of (\ref{EQW2}) and (\ref{EQW3}), setting $M(p)=2^{p/2}$.

\section{Frequent hypercyclicity vs distributional chaos}\label{SECFHCDC}
We turn to the proof of Theorem \ref{THMFHCDC}. We follow the same kind of proof.
\subsection{The sequences of integers}
For $a>1$, $\veps>0$ and $u\in\mathbb N$, we set
$$I_{u}^{a,\veps}=[(1-\veps)a^u,(1+\veps)a^u].$$
\begin{Lemma}
 There exist $a>1$ and $\veps>0$ such that $\bar d(\bigcup_{u\geq 1}I_{u}^{a,4\veps})<1$ and, for any $u>v\geq 1$,
$$I_u^{a,2\veps}\cap I_v^{a,2\veps}=\emptyset,\ I_{u}^{a,2\veps}-I_{v}^{a,2\veps}\subset I_u^{a,4\veps}.$$
\end{Lemma}
\begin{proof}
 It is easy to check that, for any $u>v\geq 1$, $I_{u}^{a,2\veps}-I_v^{a,2\veps}\subset I_{u}^{a,4\veps}$ as soon as,
for any $u\geq 2$, 
$$(1-2\veps)a^u-(1+2\veps)a^{u-1}\geq (1-4\veps)a^u.$$
This condition is satisfied provided 
$$\frac{2\veps a}{1+2\veps}\geq 1.$$
Moreover, let us also assume that $(1+4\veps)/(1-4\veps)<a$. Then 
\begin{eqnarray*}
 \bar d\left(\bigcup_{u\geq 1}I_u^{a,4\veps}\right)&\leq&\lim_{k\to+\infty}\frac{8\veps(1+\dots+a^k)}{(1+4\veps)a^k}\\
&\leq&\frac{8\veps a}{(1+4\veps)(a-1)},
\end{eqnarray*}
and this is less than 1 provided $\veps$ is small enough and $a$ is large enough. Observe also that this choice of $a$
and $\veps$ guarantees that $I_u^{a,2\veps}\cap I_v^{a,2\veps}=\emptyset$ for any $u\neq v$.
\end{proof}
From now on, we fix $a>1$ and $\veps>0$ satisfying the conclusions of the previous lemma. We then consider an increasing sequence
of positive integers $(b_p)_{p\in\mathbb N}$ such that
$$\sum_{p\geq 1}\bar d(b_p\mathbb N+[-2p,2p])+\bar d\left(\bigcup_{u\geq 1}I_u^{a,4\veps}\right)<1.$$
Observe that $b_p\geq 4p$ for any $p\geq 1$. 

We also consider a partition of $\mathbb N$ into $\bigcup_{p\geq 1}A_p$ where each $A_p$ is syndetic. For instance,
we may set $A_p=2^{p-1}\mathbb N\backslash 2^p\mathbb N$. We finally set
$$E_p=\bigcup_{u\in A_p}I_u^{a,\veps}\cap b_p\mathbb N.$$
\begin{Lemma}
 For any $p\geq 1$, $\underline d(E_p)>0$.
\end{Lemma}
\begin{proof}
 Let $(n_k)$ be an increasing enumeration of $A_p$ and let $M>0$ be such that $n_{k+1}-n_k\leq M$. Then
\begin{eqnarray*}
 \underline{d}(E_p)&\geq&\liminf_{k\to+\infty}\frac{\#E_p\big((1+\veps)a^{n_k}\big)}{a^{n_{k+1}}}\\
&\geq&\liminf_{k\to+\infty}\frac{2\veps a^{n_k}}{b_p a^{n_k+M}}>0.
\end{eqnarray*}
\end{proof}
Deleting a finite number of elements in $A_p$ if necessary, we may and shall assume that for any $u\in A_p$, 
$I_u^{a,\veps}+[-2p,2p]\subset I_u^{a,2\veps}$. Since $I_u^{a,2\veps}\cap I_v^{a,2\veps}=\emptyset$ whenever $u\neq v$ and since
$b_p\geq 4p$, we get the following lemma.
\begin{Lemma}
 Let $p,q\geq 1$, $n\in E_p$, $m\in E_q$ with $n\neq m$. Then $|n-m|>2\max(p,q)$.
\end{Lemma}
In particular, $(E_p+[-p,p])\cap (E_q+[-q,q])=\emptyset$ if $p\neq q$.

\subsection{The weight}
As in Section \ref{SUBSECWEIGHT}, we will define several weights: weights $\mathbf w^p$ such that $w_{m-n+1}^p\cdots w_0^p$ is small
when $m<n$ belong to the same $E_p$, and weights $\mathbf w^{u,v}$, $u>v$, such that $w_{m-n+1}^{u,v}\cdots w_{0}^{u,v}$
is small when $m$ belongs to $I_{v}^{a,\veps}$ and $n$ belongs to $I_u^{a,\veps}$. Elsewhere, they will be large to ensure
that $B_{\mathbf w}$ cannot be distributionally chaotic.

We begin with $\mathbf w^p$, $p\geq 1$. We set $\mathbf w^p=(w_k^p)$ any sequence of positive integers such that
$$w_{-k+1}^p\cdots w_0^p=\left\{
\begin{array}{ll}
 1&\textrm{ provided }k\notin b_p\mathbb N+[-2p,2p]\\
\frac 1{2^p}&\textrm{ provided }k\in b_p\mathbb N,
\end{array}
\right.$$
\begin{eqnarray*}
 &\frac 12\leq w_k^p\leq 2&\textrm{for any }k\in\mathbb Z\\
&w_k^p=2&\textrm{for any }k\geq 1.
\end{eqnarray*}
Let us now define $\mathbf w^{u,v}$ for $u>v$. Let $p,q\geq 1$ such that $u\in A_p$ and $v\in A_q$. We set 
$\mathbf w^{u,v}=(w_k^{u,v})$ any sequence of positive real numbers such that
$$w_{-k+1}^{u,v}\cdots w_0^{u,v}=\left\{
\begin{array}{ll}
 1&\textrm{ provided }k\notin I_u^{a,4\veps}\\
\min\left(\frac 1{2^{2p}},\frac 1{2^{2q}}\right)&\textrm{ provided }k\in I_u^{a,\veps}-I_v^{a,\veps},
\end{array}
\right.$$
\begin{eqnarray*}
 &\frac 12\leq w_k^{u,v}\leq 2&\textrm{for any }k\in\mathbb Z\\
&w_k^{u,v}=2&\textrm{for any }k\geq 1.
\end{eqnarray*}
It is possible to construct such a weight because
$$I_u^{a,\veps}-I_v^{a,\veps}+[-2\max(p,q),2\max(p,q)]\subset I_u^{a,2\veps}-I_v^{a,2\veps}\subset I_u^{a,4\veps}.$$
We finally define our weight $\mathbf w$ by setting inductively $w_{-n}$ for $n>0$ with the relation
$$w_{-n+1}\cdots w_0=\min_{p,u,v}(w_{-n+1}^p \cdots w_0^p,w_{-n+1}^{u,v}\cdots w_0^{u,v})$$
and by letting $w_k=2$ for $k\geq 1$. $\mathbf w$ is well-defined because, for a fixed $n\geq 1$, $w_{-n+1}^p\cdots w_0^p=1$ and
$w_{-n+1}^{u,v}\cdots w_{0}^{u,v}=1$ provided $p$ and $u$ are large enough. Moreover, the definition of $\mathbf w$
easily implies that $\frac12\leq w_k\leq 2$ for any $k\in\mathbb Z$,
so that the weighted shift $B_{\mathbf w}$ is bounded and invertible on $c_0(\mathbb Z)$.

\subsection{$B_{\mathbf w}$ is not distributionally chaotic}

We verify that the product $w_{-n+1}\cdots w_0$ is not small very often. 
Indeed, let $A=\mathbb N\backslash \left(\bigcup_p (b_p\mathbb N+[-2p,2p])\cup\bigcup_u I_u^{a,4\veps}\right)$.
Then our choices of $a,\veps$ and $(b_p)$ tell us that $\underline d(A)>0$. Moreover, by the construction of our weight,
$w_{-n+1}\cdots w_0=1$ provided $n\in A$. Pick now $x\in c_0(\mathbb Z)$, $x\neq 0$ and let $k$ be such that 
$x_k\neq 0$. If $n-k$ belongs to $A$, then $\|B_{\mathbf w}^{n-k}x\|_\infty\geq |x_k|>0$ so that $x$ cannot be a distributional
irregular vector for $B_{\mathbf w}$. Therefore, $B_{\mathbf w}$ is not distributionally chaotic.

\subsection{$B_{\mathbf w}$ is frequently hypercyclic}
We apply Theorem \ref{THMCARACBILATERAL}. The only thing that we do not have verified yet is property (d).
Thus, let $n\in E_p$, $m\in E_q$ with $m<n$. If $p=q$, then $m-n\in b_p\mathbb N$ so that
$$w_{m-n+1}\cdots w_0\leq\frac{1}{2^{2p}}.$$
If $p\neq q$, then there exists $u>v$ such that $n\in I_u^{a,\veps}$ and $m\in I_v^{a,\veps}$. Thus,
$$w_{m-n+1}\cdots w_0\leq\min\left(\frac1{2^{2p}},\frac{1}{2^{2q}}\right)\leq\frac1{2^{p+q}}.$$
Hence,
\begin{eqnarray*}\label{EQDC1}
 w_{m-n+1}\cdots w_0\leq\frac 1{M(p)M(q)}
\end{eqnarray*}
with $M(p)=2^p$. If $m>n$, then we just observe that $m-n\geq p+q$ to conclude
$$w_1\cdots w_{m-n}\geq 2^{p+q}=M(p)M(q).$$

\section{Final comments and open questions}\label{SECOPEN}
The work of Section \ref{SECFHCDC} shows that a frequently hypercyclic operator does not need to be
distributionally chaotic. However, it admits plenty of half distributional irregular vectors!
\begin{Prop}\label{PROPFHCDC}
Let $T\in\mathfrak L(X)$ be frequently hypercyclic. Then there exists a residual subset $\mathcal R$ of $X$ such that
any vector $y\in\mathcal R$ has a distributional unbounded orbit, namely there exists $B\subset\mathbb N$ such that
$\overline d(B)=1$ and $\lim_{n\to+\infty,n\in B}\|T^n y\|=+\infty.$
\end{Prop}
\begin{proof}
By the work of \cite{BBMP13}, it is sufficient to find $\veps>0$, a sequence $(y_k)\subset X$ and an increasing sequence
$(N_k)$ in $\mathbb N$ such that $\lim_k y_k=0$ and 
$$\#\big\{1\leq j\leq N_k;\ \|T^j y_k\|>\veps\big\}\geq \veps N_k.$$
Let $x\in FHC(T)$ and let $\eta>0$ be such that
$$\underline d\big(\big\{n\in\mathbb N; \|T^n x\|>1\big\}\big)>\eta.$$
We set $\veps=\eta/2$. For any $k\geq 1$, let $p_k>0$ be such that $\|T^{p_k}x\|<1/k$. We set
$y_k=T^{p_k}x$. This $p_k$ being fixed, we may find $N_k$ as large as we want such that
$$\#\big\{1\leq n\leq N_k;\ \|T^n y_k\|>1\big\}=\#\big\{p_k+1\leq n\leq N_k+p_k;\ \|T^n x\|>1\big\}\geq \frac{\eta N_k}2.$$
\end{proof}
This proposition has several interesting corollaries. First of all, a frequently hypercyclic operator is "almost" distributionally chaotic.
\begin{Cor}\label{CORFHCDC}
Let $T\in\mathfrak L(X)$ be frequently hypercyclic and assume that there exists a dense set $X_0\subset X$ such that
$T^n x\to 0$ for any $x\in X$. Then $T$ is distributionally chaotic.
\end{Cor}
\begin{proof}
By \cite{BBMP13}, an operator with a distributional unbounded orbit and such that $T^n x\to 0$ for any $x$ in a dense subset $X_0$ of $X$
is distributionally chaotic. And we have just proved that a frequently hypercyclic operator has a distributional unbounded orbit.
\end{proof}

Our example of a frequently hypercyclic operator which is not distributionally chaotic was a bilateral weighted shift.
This would be impossible with a unilateral weighted shift.
\begin{Cor}
A frequently hypercyclic unilateral weighted shift is distributionally chaotic.
\end{Cor}
\begin{proof}
The orbit of any vector with a finite support goes to zero, so that we may apply Corollary \ref{CORFHCDC}.
\end{proof}
Thanks to Proposition \ref{PROPFHCDC}, we can solve another open question of \cite{BBMP13}.
\begin{Def}
Let $T\in\mathfrak L(X)$. We say that the $\mathbb T$-eigenvectors of $T$ are \emph{perfectly spanning} if, for any countable set
$D\subset \mathbb T=\{z\in\mathbb C;\ |z|=1\}$, the linear span of $\bigcup_{\lambda\in\mathbb T\setminus D}\ker (T-\lambda)$ is dense in $X$.
\end{Def}
The next corollary extends a result of \cite{BBMP13} from Hilbert spaces to general Banach spaces.
\begin{Cor}
Let $T\in\mathfrak L(X)$ be  such that its $\mathbb T$-eigenvectors are perfectly spanning. Then $T$ is distributionally chaotic.
\end{Cor}
\begin{proof}
By \cite{BAYMATHERGOBEST}, $T$ is frequently hypercyclic. By \cite{BG1}, there exists a dense set $X_0\subset X$ such that
$T^n x\to 0$ for all $x\in X_0$. Thus we may apply Corollary \ref{CORFHCDC}.
\end{proof}

A striking difference between hypercyclic operators and frequently hypercyclic operators is the comparison of the size of $HC(T)$ and 
$FHC(T)$. Whereas $HC(T)$ is always residual when it is nonempty, it was shown (see for instance \cite{BM09}) that, for many frequently
hypercyclic operators, $FHC(T)$ is of first category. It turns out that $FHC(T)$ is always meagre.
\begin{Cor}\label{CORFIRST}
Let $T\in\mathfrak L(X)$ be frequently hypercyclic. Then $FHC(T)$ is a set of first category.
\end{Cor}
\begin{proof}
A vector with a distributional unbounded orbit cannot be a frequently hypercyclic vector. By Proposition \ref{PROPFHCDC},
a frequently hypercyclic operator admits a residual subset of vectors with distributionally unbounded orbit.
\end{proof}

We now turn to an interesting problem regarding frequently hypercyclic operators.
\medskip

\noindent {\bf Question.} Let $T\in \mathfrak L(X) $ be frequently hypercyclic and invertible. Is $T^{-1}$ invertible?\medskip

In view of this paper, it is natural to study whether a bilateral weighted shift on $c_0$ could be a counterexample.
It we look at Theorem \ref{THMCARACBILATERAL}, this does not seem impossible; indeed, because of (c),
the conditions on the right part and on the left part of $B_{\mathbf w}$ are not symmetric. However, it could be possible
that this condition is superfluous.\medskip

\noindent{\bf Question.} Let $\mathbf w=(w_n)_{n\in\mathbb Z}$ be a bounded and bounded below sequence
of positive integers such that conditions (a), (b) and (d) of Theorem \ref{THMCARACBILATERAL}
are satisfied. Does it automatically satisfy conditions (a), (b), (c) and (d), maybe for another family
$(E_p)$ of subsets of $\mathbb N$? 

\medskip

At least, we are able to give a partial positive answer to our first question.
\begin{Prop}
Let $T\in\mathfrak L(X)$ be frequently hypercyclic and invertible. Then $T^{-1}$ is
$\mathcal U$-frequently hypercyclic.
\end{Prop}
\begin{proof}
Let $x\in FHC(T)$, let $(U_k)$ be a basis of open subsets of $X$ and let 
$$\delta_k=\underline d\big(\big\{n\in\mathbb N:\ T^n x\in U_k\big\}\big).$$
We set, for $k,N\geq 1$ and $n\geq N$, 
$$U_{k,N,n}=\big\{y\in X;\ \#\{1\leq j\leq n;\ T^{-j}y\in U_k\}\geq \delta_kn/2\big\}.$$
$U_{k,N,n}$ is clearly open. Moreover, $\bigcap_{k,N\geq 1}\bigcup_{n\geq N} U_{k,N,n}$ contains $\mathcal UFHC(T^{-1})$.
We intend to apply Baire's theorem and we prove that $\bigcup_{n\geq N} U_{k,N,n}$ is dense for any $k,N\geq 1$. 
Let $V\subset X$ be open and nonempty and let $N_k\geq N$ be such that, for any 
$n\geq N_k$, 
$$\#\{0\leq j\leq n;\ T^j x\in U_k\}\geq\frac{\delta_k n}2.$$
There exists $n\geq N_k$ such that $T^n x\in V$. Let us set $y=T^n x$. Then 
$$\#\{0\leq j\leq n;\ T^{-j} y\in U_k\}=\#\{0\leq j\leq n;\ T^j x\in U_k\}.$$
In particular, $y\in \bigcup_{n\geq N}U_{k,N,n}\cap V$. By Baire's theorem, $\mathcal UFHC(T^{-1})$ is residual, hence
nonempty.
\end{proof}

An easy modification of the previous argument yields the following interesting corollary, to be compared with Corollary \ref{CORFIRST}.
\begin{Prop}
Let $T\in\mathfrak L(X)$ be $\mathcal U$-frequently hypercyclic. Then $\mathcal UFHC(T)$ is residual.
\end{Prop}
\begin{proof}
Let $x\in \mathcal UFHC(T)$, let $(U_k)$ be a basis of open subsets of $X$ and let 
$$\delta_k=\bar d\big(\big\{n\in\mathbb N:\ T^n x\in U_k\big\}\big).$$
We set, for $k,N\geq 1$ and $n\geq N$, 
$$U_{k,N,n}=\big\{y\in X;\ \#\{1\leq j\leq n;\ T^{j}y\in U_k\}\geq \delta_kn/2\big\},$$
which is open. Moreover, it is easy to see that any iterate $T^p x$ belongs to $\bigcap_{k,N\geq 1}\bigcup_{n\geq N}U_{k,N,n}$. Since
these iterates are dense in $X$, $\bigcap_{k,N\geq 1}\bigcup_{n\geq N}U_{k,N,n}$ is a dense $G_\delta$-set, which is contained in $\mathcal UFHC(T)$.
\end{proof}

\bigskip

\noindent \textsc{Acknowledgements.} We thank Pr. Vitaly Bergelson for fruitful discussions.

 \end{document}